\documentclass[11pt]{article}
\usepackage{amsfonts,amsmath, amssymb,latexsym}
\usepackage[OT2,OT1]{fontenc}
\usepackage[all,cmtip]{xy}
\def\SH{\mbox{\fontencoding{OT2}\selectfont\char88}}

\def\Z{{\mathbb Z}}

\def\GL{{\rm GL}}
\def\PGL{{\rm PGL}}

\def\Stab{{\rm Stab}}
\def\Sym{{\rm Sym}}

\def\SO{{\rm SO}}
\def\P{{\mathbb P}}

\def\Aut{{\rm Aut}}
\def\irr{{\rm irr}}

\def\Inv{{\rm Inv}}
\def\red{{\rm red}}
\def\Vol{{\rm Vol}}
\def\R{{\mathbb R}}
\def\F{{\mathbb F}}
\def\FF{{\mathcal F}}

\def\Q{{\mathbb Q}}

\def\nv{{V_\Z^{(1)}}}

\def\pnv{{V_\Z^{(i)}}}
\def\prs{{V_\R^{(0\#)}}}

\def\tr{{V_\R^{(2)}}}
\def\pnr{{V_\R^{(i)}}}

\def\H{{\mathcal H}}

\def\J{{\mathcal J}}
\def\C{{\mathcal C}}

\def\W{{\mathcal W}}

\def\Z{{\mathbb Z}}
\def\P{{\mathbb P}}
\def\F{{\mathbb F}}
\def\Q{{\mathbb Q}}
\def\C{{\mathbb C}}

\def\H{{\mathcal H}}

\newtheorem{theorem}{Theorem}
\newtheorem{corollary}[theorem]{Corollary}
\newtheorem{conjecture}[theorem]{Conjecture}
\newtheorem{lemma}[theorem]{Lemma}
\newtheorem{proposition}[theorem]{Proposition}
\newenvironment{proof}{\noindent {\bf Proof:}}{$\Box$ \vspace{2 ex}}

\title{The average number of elements in the $4$-Selmer groups \\of
  elliptic curves is 7}

\author{Manjul Bhargava and Arul Shankar}
\begin{document}
\maketitle

% \begin{abstract}
%   We prove that when all elliptic curves over $\Q$ are ordered by
%   height, the average size of their 4-Selmer groups is equal to 7.  As
%   a consequence, we show that a positive proportion of 2-Selmer
%   elements of elliptic curves, when ordered by height, do not lift to
%   4-Selmer elements, and thus correspond to nontrivial 2-torsion
%   elements in the associated Tate--Shafarevich groups.
% \end{abstract}

\section{Introduction}

Any elliptic curve $E$ over $\Q$ is isomorphic to a unique curve of
the form $E_{A,B}:y^2=x^3+Ax+B$, where $A,B \in \Z$ and for all primes
$p$:\, $p^6 \nmid B$ whenever $p^4 \mid A$. Let $H(E_{A,B})$ denote
the (naive) $height$ of $E_{A,B}$, defined by $H(E_{A,B}):= \max
\{4|A^3|,27B^2\}$.  

In previous papers (\cite{BS} and \cite{TC}), we showed that the
average size of the 2-Selmer group of all elliptic curves over $\Q$, when
ordered by height, is 3; meanwhile the average size of the 3-Selmer group
is~4.
The purpose of this article is to prove an analogous result for the
average size of the $4$-Selmer group of all elliptic
curves over $\Q$. Specifically, we prove the following theorem:

\begin{theorem}\label{mainellip}
When all elliptic curves $E/\Q$ are ordered by height,
the average size of the $4$-Selmer group $S_{4}(E)$ is equal to $7$.
\end{theorem}
We will in fact prove a stronger version of Theorem \ref{mainellip}
where we compute the average size of the $4$-Selmer group of elliptic
curves satisfying any finite set of congruence conditions:

\begin{theorem}\label{ellipcong}
  When elliptic curves $E:y^2=x^3+Ax+B$ over $\Q$, in any family
  defined by finitely many congruence conditions on the coefficients
  $A$ and $B$, are ordered by height, the average size of the
  $4$-Selmer group $S_4(E)$ is~$7$.
\end{theorem}
We will also prove an analogue of Theorem~\ref{ellipcong} for certain
families of elliptic curves defined by infinitely many congruence
conditions (e.g., the family of all semistable elliptic curves). 

Since we have shown in \cite{BS} that the average number of elements in
the 2-Selmer groups of elliptic curves over $\Q$ is 3,
we may use Theorem~\ref{mainellip} to prove that a positive proportion of 2-Selmer
elements of
elliptic curves do {\it not} lift to 4-Selmer elements:

\begin{theorem}\label{4lift}
For an elliptic curve $E$ over $\Q$, let $\times2:S_4(E)\to S_2(E)$ denote
the multiplication-by-$2$ map. Then, when elliptic curves $E$ over $\Q$ are
ordered by height, the average number elements in the $2$-Selmer group of $E$
that have no preimage under $\times2$ is at least $3/5>0$.
\end{theorem}
It follows, in particular, that a positive proportion (in fact, at
least one fifth) of all 2-Selmer
elements of elliptic curves $E$ over $\Q$, when such $E$ are
ordered by height, correspond to nontrivial 2-torsion elements of the
Tate--Shafarevich group \SH$_E$ of $E$.  Another consequence is that
there exist infinitely many elliptic curves $E$ over $\Q$ with trivial
rational 2-torsion for which the $2$-primary part of the group
{\SH}$_E$ contains 
$\Z/2\Z$ as a factor.

As we will explain, Theorems~\ref{mainellip} and \ref{ellipcong}, and
the methods of their proofs, lead naturally to the following
conjecture on the average size of the $n$-Selmer group of elliptic
curves for general $n$:

\begin{conjecture}\label{sigmaconj}
  Let $n$ be any positive integer.  Then, when all elliptic curves $E$
  are ordered by height, the average size of the $n$-Selmer group
  $S_n(E)$ is $\sigma(n)$, the sum of the divisors of~$n$.
\end{conjecture}
Thus the conjecture is proven for $n=2$, $n=3$, and $n=4$ (and also
for $n=1$!).  We will prove Conjecture~\ref{sigmaconj} for $n=5$ in
\cite{5sel}.  This paper represents the first time that the average
size of the $n$-Selmer group has been determined for a {composite
  value of $n$.

Conjecture~\ref{sigmaconj} also has consequences
  for the distribution of ranks of elliptic curves.  Since $\epsilon
  n^2$ grows faster than $\sigma(n)$, as a function of $n$, for any
  $\epsilon>0$, we obtain:

\begin{proposition}\label{rankdist}
  Suppose that Conjecture~$\ref{sigmaconj}$ is true for all $n$, or
  indeed, any infinite sequence of positive integers~$n$. Then
when all
  elliptic curves over $\Q$ are ordered by height, 
a density of $100\%$ have rank $\leq 1$. 
\end{proposition}
The parity conjecture states that an elliptic curve has even
rank if and only if its root number is~$1$. Hence the above proposition has
the following consequence:
\begin{corollary}
  Suppose that Conjecture~$\ref{sigmaconj}$ is true for all $n$, or
  any infinite sequence of positive integers~$n$. Further assume that
  the root numbers of elliptic curves are equidistributed and that the
  parity conjecture holds. Then when elliptic curves are ordered by
  height, $50\%$ have rank $0$ and $50\%$ have rank $1$.
\end{corollary}
Thus our results on Selmer groups above give independent theoretical
evidence for the elliptic curve rank distribution conjecture, due to
Goldfeld~\cite{G1} and Katz--Sarnak~\cite{KS} (see also~\cite{BMSW}
for a nice survey), which states that 50\% of all elliptic curves
have rank 0 and 50\% rank 1.

Our method for proving Theorem~1 is as follows.  We view $n$-Selmer
elements of an elliptic curve $E$ as locally soluble $n$-coverings of
$E$.  Here, an {\it $n$-covering of $E$} is a genus one curve $C/\Q$
together with maps $\phi:C\to E$ and $\theta:C\to E$, where $\phi$ is
an isomorphism defined over $\C$, and $\theta$ is a degree $n^2$ map
defined over $\Q$, such that the following diagram commutes:
$$\xymatrix{E \ar[r]^{[n]} &E\\C\ar[u]^\phi\ar[ur]_\theta}$$ An
$n$-covering $C$ is said to be {\it locally soluble} if $C$ has points
defined over $\R$ and over $\Q_p$ for all primes $p$.  Cassels
\cite{Cassels} proved that any locally soluble $n$-covering has a
degree $n$ divisor defined over $\Q$, yielding an embedding of $C$
into $\P^{n-1}$ defined over $\Q$. We may thus represent $n$-Selmer
elements of elliptic curves as genus one normal curves in $\P^{n-1}$.  When
$n=4$, as is well known, any such genus one curve in $\P^{n-1}=\P^3$
arises naturally as the complete intersection of a pair of quadrics in
$\P^3$, where the two quadrics are well-defined up to appropriate 
changes-of-basis. 
Indeed, it turns out that 4-Selmer elements of an
elliptic curve $E_{A,B}$ over $\Q$ may naturally be viewed in terms of 
the ``locally soluble'' orbits of $G_\Q$ on $V_\Q$, where $G$ is the 
algebraic group such that
$$G_R:=\{(g_2,g_4)\in\GL_2(R)\times\GL_4(R):\det(g_2)\det(g_4)=1\}/
\{(\lambda^{-2}I_2,\lambda I_4):\lambda\in R^\times\}$$ for all rings
$R$, and $V$ is the representation $2\otimes \Sym^2(4)$ of pairs of
quadrics (see~\cite[\S4.3]{BhHo} for the reasons behind this choice of group $G_R$).  
The invariant ring for the
representation of $G_\C$ on $V_\C$ turns out to be freely generated by
two invariants, which naturally correspond to the invariants~$A$
and~$B$ \pagebreak of the
Jacobian elliptic curve $E_{A,B}$ of the associated genus one curve in~$\P^3$.  These
classical connections among orbits on pairs of quadrics, genus one
normal curves in~$\P^3$, and explicit 4-descent on elliptic curves
over global fields were fully developed in recent years in a series of
beautiful works by An, Kim, Marshall, Marshall, McCallum, and
Perlis~\cite{MMM}, Siksek~\cite{SikThesis}, Merriman, Siksek, and
Smart~\cite{MSS}, Womack~\cite{Womack}, and
Fisher~\cite{Fisher1,Fisher2}.  Furthermore, it is a theorem of
Cremona, Fisher, and Stoll~\cite{CFS} that the orbit of $G_\Q$ on
$V_\Q$ corresponding to any 4-Selmer element of $E_{A,B}$
always contains an element of $V_\Z$ having 
invariants exactly $A$ and $B$ (up to bounded powers of 2 and 3).

To prove Theorem~\ref{mainellip}, we are thus reduced to counting
suitable orbits of $G_\Z$ on $V_\Z$, where a counting method
involving the geometry-of-numbers,
developed in \cite{dodqf}, \cite{dodpf}, and \cite{BS}, may be applied. The method involves
counting lattice points, in fundamental domains for the action of $G_\Z$ on
$V_\R$, corresponding to elliptic curves of bounded height.  The
difficulty, as in \cite{BS}, 
lies in dealing with the cusps of these
fundamental domains.  In the case at hand, a number of
suitable adaptations to the method of \cite{BS} are required.
For example, the geometry of the cusps of the fundamental domains is
considerably more complicated than that in \cite{BS}.  In addition,
the method requires a count of elements having squarefree
discriminant, which again necessitates a technique that is quite
different than that used in \cite{BS} (but is closer to that
used in \cite{TC}); this is perhaps the most technical ingredient of
the paper.  

The end result of the method, however, is quite simple to state.
Namely, we show that the average occurring in Theorem~\ref{mainellip}
arises naturally as the sum of {two contributions}. One comes from the
main body of the fundamental domains, which corresponds to the average
number of elements in the 4-Selmer group having exact order 4; we show
that this average is given by the Tamagawa number
$\tau(G_\Q)=\tau(\PGL_4(\Q))=4$. The other comes from the {cusps} of
the fundamental domains, which corresponds to the average number of elements
in the 4-Selmer group having order strictly less than~4.  This latter contribution
is equal to the average size of the 2-Selmer group, which is 3 by
the work of \cite{BS}. The sum {$4 + 3 = 7$} then yields the average
size of the 4-Selmer group, as stated in Theorem~1.  (This also
explains why, in general, we expect the average size of the $n$-Selmer
group to be $\sigma(n)$. Namely, by the analogous reasoning, we 
expect the average number of order $n$ elements in the $n$-Selmer
group to equal $n$, the Tamagawa number of $\PGL_n$; summing
over the divisors of $n$ yields Conjecture~\ref{sigmaconj}.)

In Section~2, we recall the parametrization of elements of the
4-Selmer groups of elliptic curves by orbits of 
$G_\Z$ on $V_\Z=\Z^2\otimes\Sym^2(\Z^4)$,
collecting the necessary results from \cite{MMM}, \cite{BhHo}, and
especially \cite{CFS}.  In Section 3, we then adapt the methods of
\cite{dodpf} and \cite{BS} to count the number of
$G_\Z$-orbits on $V_\Z$ of bounded height.  In Section 4, by
developing a suitable sieve, we then count just those elements that
correspond to 4-Selmer elements of exact order 4 in appropriate
congruence families of elliptic curves having bounded height.
Combined with the average size of the 2-Selmer group in such
congruence families as determined in \cite{BS}, this is then used to
deduce Theorems~\ref{mainellip}, \ref{ellipcong}, and \ref{4lift}.

\section{Pairs of quaternary quadratic forms and $4$-coverings of
  elliptic curves}

Let $E:y^2=x^3+Ax+B$ be an elliptic curve over $\Q$, where $A$ and $B$
are integers such that, for all primes $p$, we have $p^6\nmid B$ if $p^4\mid A$. We define the quantities $I(E)$ and $J(E)$ by
\begin{equation}\label{eqEIJ}
  \begin{array}{rcl}
    I(E)\!\!\!\!&:=&\!\!\!\!-3A,\\[.02in]
    J(E)\!\!\!\!&:=&\!\!\!\!-27B.
  \end{array}
\end{equation}

In this section, we collect results relating $4$-coverings of elliptic
curves to certain orbits on pairs of quaternary quadratic
forms. For our applications, we need to consider not just
elliptic curves over~$\Q$, but also elliptic curves over other fields
such as $\R$ and $\Q_p$.
For any ring $R$ of characteristic~0 (or prime to~6), 
let $V_R$ denote the space of pairs of quaternary
quadratic forms with coefficients in~$R$. We always identify
quadratic forms with their Gram-matrices, and write elements $(A,B)\in
V_R$ as pairs of $4\times 4$ symmetric matrices with
\begin{equation}\label{eqAB}
2\cdot(A,B)=\left( \left[ \begin{array}{cccc} 2a_{11} & a_{12} & a_{13} & a_{14}\\ a_{12} & 2a_{22} & a_{23} & a_{24}\\ a_{13} & a_{23} & 2a_{33} & a_{34}\\ a_{14} & a_{24} & a_{34} & 2a_{44}\end{array} \right],
\left[ \begin{array}{cccc} 2b_{11} & b_{12} & b_{13} & b_{14}\\ b_{12} & 2b_{22} & b_{23} & b_{24}\\ b_{13} & b_{23} & 2b_{33} & b_{34}\\ b_{14} & b_{24} & b_{34} & 2b_{44}\end{array} \right]\right)
\end{equation}
where $a_{ij}$ and $b_{ij}$ are elements of $R$.

The group $\GL_2(R)\times\GL_4(R)$ acts naturally on $V_R$: an
element $g_2=\bigl(\begin{smallmatrix} r & {s}\\ {t}&
  u\end{smallmatrix}\bigr)\in \GL_2(R)$ acts via $g_2\cdot
(A,B)=(rA+sB,tA+uB)$ while an element $g_4\in \GL_4(R)$ acts via
$g_4\cdot (A,B)=(g_4Ag_4^t,g_4Bg_4^t)$.  It is clear that the actions
of $g_2$ and $g_4$ commute. Also note that the element
$(\lambda^{-2}I_2,\lambda I_4)$ acts trivially on $V_R$, where
$\lambda\in R^\times$ and $I_n$ denotes the identity element in
$\GL_n(R)$. We thus obtain a faithful action of $G_R$ on $V_R$, where $G_R$ is the group
\begin{equation}\label{eqGR}
  G_R:=\{(g_2,g_4)\in\GL_2(R)\times\GL_4(R):\det(g_2)\det(g_4)=1\}/\{(\lambda^{-2}I_2,\lambda I_4):\lambda\in R^\times\}.
\end{equation}

We now describe the ring of invariants for the action of $G_\C$ on
$V_\C$. If $(A,B)\in V_\C$, we define the {\it binary quartic
  resolvent form} $f_{A,B}$ of $(A,B)$ by
\begin{equation}\label{eqfAB}
  f_{A,B}(x,y):=2^4\det(Ax+By).
\end{equation}
If $(A',B')=(g_2,g_4)\cdot(A,B)$ for $(g_2,g_4)\in G_\C$, then one
checks the identity
\begin{equation}\label{twistdef}
  f_{A',B'}(x,y)=\det(g_4)^2f_{A,B}((x,y)\cdot
  g_2)=\frac{f_{A,B}((x,y)\cdot g_2)}{\det(g_2)^2}.
\end{equation}
The action of $\PGL_2(\C)$ on the space of binary quartic forms over
$\C$, defined by (\ref{twistdef}), has a free ring of invariants,
generated by two elements traditionally denoted by $I$ and $J$. (See,
e.g., \cite[Equation~(4)]{BS} for the definitions of $I$ and $J$.)
Thus the quantities $I$ and $J$ defined by
\begin{equation}\label{eqABIJ}
  \begin{array}{rcl}
    I(A,B) &:=& I(f_{A,B})\\
    J(A,B) &:=& J(f_{A,B})
  \end{array}
\end{equation}
are also invariant, under the action of $G_\C$ on $V_\C$, and in fact they freely
generate the full ring of invariants for this action.  We may use the above
definitions of $f_{A,B}$, $I(A,B)$, and $J(A,B)$ for elements $(A,B)\in
V_R$, where $R$ is any ring. Note that since $I(f)$ and $J(f)$ are
polynomials having degrees $2$ and $3$, respectively, in the coefficients
of $f$, the polynomials $I(A,B)$ and $J(A,B)$ have degrees $8$ and
$12$, respectively, in the coefficients of $(A,B)$.

The significance of the action of $G_R$ on $V_R$ may be seen from the
following three propositions.
For a field $K$, we say that $(A,B)\in V_K$ is {\it $K$-soluble} if
the quadrics defined by $A$ and $B$ have a $K$-rational point of intersection in
$\P^3$. Then we have:
\begin{proposition}\label{propparamfield}
  Let $K$ be a field having characteristic not $2$ or $3$. Let
  $E:y^2=x^3-\frac{I}3x-\frac{J}{27}$ be an elliptic curve over $K$.
  Then there exists a bijection between elements in $E(K)/4E(K)$ and
  $G_K$-orbits of $K\!$-soluble elements in $V_K$ having invariants
  equal to $I$ and $J$. Under this bijection, a $G_K$-orbit $G_K\cdot(A,B)$ corresponds to
  an element in $E(K)/4E(K)$ having order less than $4$ if and only if
  the binary quartic resolvent form of $(A,B)$ has a linear factor
  over $K$.

  Furthermore, the stabilizer in $G_K$ of any $($not necessarily
  $K\!$-soluble$)$ element in $V_K$, having nonzero
  discriminant and invariants $I$ and $J$, is isomorphic to $E(K)[4]$,
  where $E$ is the elliptic curve defined by
  $y^2=x^3-\frac{I}{3}x-\frac{J}{27}$.
\end{proposition}
\begin{proof}
  The first and third assertions of the proposition, concerning the
  bijection and the stabilizer, follow immediately
  from \cite{MMM} and \cite[\S 4.3]{BhHo}. For the second assertion,
  regarding the elements of $E(K)/4E(K)$ having order less than 4, 
  \cite[\S3.3]{MMM} states that if $C_4\to E$ is the $4$-covering of $E$
  corresponding to $(A,B)$, then it factors through a 2-covering $C_2$
  of $E$, i.e., we
  have maps $C_4\to
  C_2\to E$, 
where $C_2\to E$ is the $2$-covering corresponding to the binary
quartic resolvent form of $(A,B)$ via \cite[\S3.1]{MMM}. Hence $(A,B)$
corresponds to an element having order less than $4$ if and only if
its binary quartic resolvent form corresponds to a trivial element in
$E(K)/2E(K)$, i.e., it has a linear factor over $\Q$ \cite[Proposition
  2.2]{CS}.
\end{proof}

An element $(A,B)\in V_\Q$ is said to be {\it locally soluble} if it
is $\R$-soluble and $\Q_p$-soluble for all primes $p$. We
similarly then obtain the following proposition:
\begin{proposition}\label{propparamq}
  Let $E:y^2=x^3-\frac{I}{3}x-\frac{J}{27}$ be an elliptic curve over
  $\Q$. Then there exists a bijection between elements in the
  $4$-Selmer group of $E$ and $G_\Q$-orbits on locally soluble
  elements in $V_\Q$ having invariants equal to $I$ and $J$.

  Furthermore, if $(A,B)$ has invariants $I$ and $J$, then the
  $G_\Q$-orbit $G_\Q\cdot(A,B)$ corresponds to an element in $S_4(E)$
  having order less than $4$ if and only if the binary quartic
  resolvent form of $(A,B)$ has a rational linear factor.
\end{proposition}

By the work of Cremona, Fisher, and Stoll~\cite[Theorem 1.1]{CFS}, any
locally soluble element $(A,B)\in V_\Q$ having integral invariants $I$
and $J$ is $\GL_2(\Q)\times\GL_4(\Q)$-equivalent to an integral
element $(A',B')\in V_\Z$ having the same invariants $I$ and $J$. In
particular, it follows that such an $(A,B)$ is $G_\Q$-equivalent to
either $(A',B')$ or $(A',-B')$. Since $(A',B')$ and $(A',-B')$ have
the same invariants, we obtain the following proposition:
\begin{proposition}\label{propselparz}
  Let $E/\Q$ be an elliptic curve. Then the elements in the $4$-Selmer
  group of $E$ are in bijective correspondence with $G_\Q$-equivalence
  classes on the set of locally soluble elements in $V_\Z$ having
  invariants equal to $I(E)$ and $J(E)$. 

Furthermore, under this correspondence, elements of exact order $4$
correspond to the $G_\Q$-equivalence classes whose binary quartic
resolvent forms have no rational linear factor.
\end{proposition}

Motivated by Propositions~\ref{propparamfield}--\ref{propselparz}, we
say that an element of $V_\Z$ (or $V_\Q$) is {\it strongly irreducible}
if its binary quartic resolvent form does not possess a rational linear factor.
Thus to count the number of 4-Selmer elements of elliptic curves
having bounded invariants, we wish to count the number of
$G_\Q$-equivalence classes of strongly irreducible elements in $V_\Z$
having bounded invariants.  In the next section, we begin by first determining
the asymptotic number of $G_\Z$-equivalence classes.

\section{The number of $G_\Z$-classes of strongly irreducible pairs of integral quaternary quadratic forms having bounded invariants}

For $i\in\{0,1,2\}$, let $V_\R^{(i)}$ denote the set of 
elements $(A,B)\in V_\R$ such that the binary quartic resolvent form
$f_{A,B}(x,y):=2^4\det(Ax+By)$ has nonzero discriminant, $i$ pairs
of complex conjugate roots in $\P^1_\C$, and thus $4-2i$ roots in
$\P^1_\R$. It
follows from \cite[Lemma 6.2.2]{SikThesis} that every element in
$V_\Z^{(1)}$ and $V_\Z^{(2)}$ is $\R$-soluble. However, this is not
the case for all elements in $V_\Z^{(0)}$; we denote the set of
$\R$-soluble elements in $V_\R^{(0)}$ by $V_\R^{(0\#)}$. 
Let $V_\Z^{(i)}:=V_\Z\cap V_\R^{(i)}$ for $i\in \{0,1,2,0\#\}$.
Then the action of $G_\Z$ on $V_\Z$ preserves also the sets $V_\Z^{(i)}$. 

The invariants $I(A,B)$ and $J(A,B)$ of $(A,B)\in V_\Z$ are as defined 
in (\ref{eqABIJ}). We then define the {\it
  discriminant} and the {\it height} of $(A,B)$ having
invariants $I$ and $J$ as follows:
\begin{equation}\label{eqABDH}
\begin{array}{ccccccc}
  \Delta(A,B)&:=&\Delta(f_{A,B})&=&\Delta(I,J)&:=&(4I^3-J^2)/27;\\
  H(A,B)&:=& H(f_{A,B}) &:=&H(I,J)&:=&\max\{|I^3|,J^2/4\}.
\end{array}
\end{equation}
Equation (\ref{eqABDH}) yields an expression for the discriminant
$\Delta(A,B)$ that is an integer polynomial of degree 24 in the
entries of $A$ and $B$.  We use (\ref{eqABDH}) as the definition of
the discriminant of elements in $V_R$ for any ring $R$, and as the
definition of the height of elements in $V_\R$.

Our purpose in this section is to count the number of strongly
irreducible $G_\Z$-orbits on $\pnv$ having bounded height for
$i\in\{0\#,1,2\}$.  To state the precise result we need some further
notation.  For any $G_\Z$-invariant set $S\subset V_\Z$, let $N(S;X)$
denote the number of $G_\Z$-equivalence classes on $S^\irr$  
having height less than $X$, where $S^\irr$ is used to denote the set of strongly
irreducible elements of $S$.  Let
$N^{+}(X)$ (resp.\ $N^{-}(X)$) denote the number of integer pairs
$(I,J)$ satisfying $\Delta(I,J)>0$ (resp.\ $\Delta(I,J)<0$) and
$H(I,J)<X$. 
By~\cite[Proposition~2.10]{BS}, we have 
\begin{equation}\label{eqnij}
  \begin{array}{rcl}
    N^{\pm}(X)&=&\:\displaystyle\frac85\:X^{5/6}+O(X^{1/2}),\\[0.175in]
    N^{\pm}(X)&=&\displaystyle\frac{32}5X^{5/6}+O(X^{1/2}).
  \end{array}
\end{equation}
Let $\omega$ be a fixed algebraic nonzero top-degree left-invariant
differential form on $G$ such that, for every prime $p$, the measure
of $G_{\Z_p}$ computed with respect to $\omega$ is
$\#G_{\F_p}/p^{\text{dim} G}=\#G_{\F_p}/p^{18}$. There is a natural
map $G_\R\times R^{(i)}\to V_\R^{(i)}$ given by
$(\gamma,x)\mapsto\gamma\cdot x$, where the sets $R^{(i)}\subset V_\R$
are defined just after (\ref{eqR}). We will see in Section 3.3 that the
Jacobian change of variables of this map (computed with respect to the measure on
$G_\R$ obtained from $\omega$, the measure $dIdJ$ on $R^{(i)}$, and
the Euclidean measure on $V_\R$ normalized so that $V_\Z$ has
covolume~$1$) is a nonzero rational constant independent of
$i$. Henceforth, we will denote this constant by $\J$.

The aim of this section is to prove the following theorem:

\begin{theorem}\label{thsec2main}
We have:
\begin{itemize}
\item[{\rm (a)}]$N(\nv;X)=\displaystyle\frac14|\J|\cdot\Vol(G_\Z\backslash G_\R)N^{-}(X)+o(X^{5/6});$
\item[{\rm
    (b)}]$N(\pnv;X)=\,\displaystyle\frac18|\J|\cdot\Vol(G_\Z\backslash
  G_\R)N^{+}(X)+o(X^{5/6})$ for $i=0\#$ and $2$, 
\end{itemize}
where the volume of $G_\Z\backslash G_\R$ is
computed with respect to the measure obtained from $\omega$.
\end{theorem}
The value of $\J$ is not difficult to compute, but is irrelevant for
the proofs of Theorems~\ref{mainellip} and~\ref{ellipcong} because of
its cancellation in (\ref{eqJcance}).

\subsection{Reduction theory}\label{s31}
In this subsection, we construct certain finite covers of fundamental
domains for the action of $G_\Z$ on $\pnr$ for $i\in\{0\#,1,2\}$.  We
start by constructing fundamental sets for the action of $G_\R$ on
$\pnr$, for $i\in\{0\#,1,2\}$. The following result is a consequence
of Proposition \ref{propparamfield} along with the fact that every
element in $V_\R^{(0+)}$, $V_\R^{(1)}$ and $V_\R^{(2)}$ is $\R$-soluble.
\begin{proposition}\label{propgrvr}
  Let $(I,J)$ be an element of $\R\times\R$ such that $\Delta(I,J)\neq 0$. Then
  \begin{itemize}
  \item[{\rm (1)}] If $\Delta(I,J)<0$, then the set of elements in
    $V_\R$ having fixed invariants $I$ and $J$ consists of one
    $\R$-soluble $G_\R$-orbit. The size of the stabilizer in $G_\R$ of
    any element in this orbit is $4$.
  \item[{\rm (2)}] If $\Delta(I,J)>0$, then the set of $\R$-soluble
    elements in $V_\R$ having fixed invariants $I$ and $J$ consists of
    two $G_\R$-orbits. There is one such orbit from each of $\prs$ and
    $\tr$. The size of the stabilizer in $G_\R$ of any element in
    either of these orbits is $8$.
  \end{itemize}
\end{proposition}

For $i=0\#$, $1$, and $2$, we choose fundamental sets $R^{(i)}\subset
V_\R^{(i)}$ for the action of $G_\R$ on $V_\R^{(i)}$ as follows.  Let
$f_{I,J}^{(i)}$ be the forms constructed in \cite[Table 1]{BS}, for
$i=0$, $1$, and $2$.  Then for each $(I,J)\in\R\times\R$ with
$\Delta(I,J)>0$ (resp.\ $\Delta(I,J)<0$) and $H(I,J)=1$, we obtain two
binary quartic forms $f_{I,J}^{(0)}$ and $f_{I,J}^{(2)}$ (resp.\ one
binary quartic form $f_{I,J}^{(1)}$) having invariants $I$ and
$J$. The coefficients of all these forms $f_{I,J}^{(i)}$ are bounded
independently of $I$ and $J$. We write
\begin{equation}
\begin{array}{rcl}
f_{I,J}^{(0)}&=&\kappa y(x+\lambda_1y)(x+\lambda_2y)(x+\lambda_3y),\\[.02in]
f_{I,J}^{(1)}&=&\kappa y(x+\lambda y)(x^2+r^2y^2 ),\\[.02in]
f_{I,J}^{(2)}&=&\kappa(x^2+r_1^2y^2)(x^2+r_2^2y^2),
\end{array}
\end{equation}
with $\kappa>0$, $\lambda_1>\lambda_2>\lambda_3$, $r>0$ and $r_1>r_2>0$.

Consider the sets
\begin{equation}\label{eqR}
  \begin{array}{ccl}
L^{(0\#)}&=&\left\{\kappa^{1/4}
\left( \left[ \begin{array}{cccc} 0 &  &  & \\  & -1 &  & \\  &  & 1 & \\  &  & & -1 \end{array} \right],
\left[ \begin{array}{cccc} 1 &  &  & \\  & \!\!\!\!-\lambda_1 &  & \\  &  & \lambda_2 & \\  &  &  & -\lambda_3\end{array} \right]\right)
\right\},\\[.35in]
L^{(1)}&=&\left\{\kappa^{1/4}
\left( \left[ \begin{array}{cccc} 0 &  &  & \\  & -1 &  & \\  &  &  & 1\\  &  & 1 &  \end{array} \right],
\left[ \begin{array}{cccc} 1 &  &  & \\  & -\lambda &  & \\  &  & r & \\  &  &  & -r\end{array} \right]\right)
\right\},\\[.35in]
L^{(2)}&=&\left\{\kappa^{1/4}
\left( \left[ \begin{array}{cccc}  & 1 &  & \\ 1 &  &  & \\  &  &  & 1\\  &  & 1 &  \end{array} \right],
\left[ \begin{array}{cccc} r_1 &  &  & \\  & -r_1 &  & \\  &  & r_2 & \\  &  &  & -r_2\end{array} \right]\right)
\right\}.
  \end{array}
\end{equation}
%The sets $L^{(i)}$ then yield fundamental regions for the action of $G_\R$ on the elements of $V^{(i)}$ 
%having height 1, for $i=0\#$, $1$, and $2$.  
Since the coefficients of the forms $f^{(i)}_{I,J}$ are
bounded independently of $I$ and $J$, the coefficients of the elements
in $L^{(0\#)}$, $L^{(1)}$, and $L^{(2)}$ are also bounded independently
of $I$ and $J$.

Let $R^{(i)}$ be defined to be $\R_{>0}\cdot L^{(i)}$. The sets
$R^{(i)}$ then satisfy
the following two properties that we use throughout this
section:
\begin{enumerate}
\item The sets $R^{(i)}$ are subsets of $V^{(i)}_\R$ for $i=0\#$, $1$,
  and $2$. Furthermore, $R^{(0\#)}$ and $R^{(2)}$ (resp.\ $R^{(1)}$)
  contain exactly one point having invariants $I$ and $J$ for each
  pair $(I,J)\in\R\times\R$ with $\Delta(I,J)>0$ (resp.\
  $\Delta(I,J)<0$).
\item For $i\in\{0\#,1,2\}$, the coefficients of all the elements
  of height $X$ in $R^{(i)}$ are bounded by~$O(X^{1/24})$.
\end{enumerate}
To verify that $R^{(i)}\subset V^{(i)}_\R$, it suffices to show that
the elements in $L^{(i)}$ are soluble over $\R$. For $(A,B)\in
L^{(0\#)}$, this follows by applying~\cite[Theorem~6.3.1]{SikThesis}
on $(A+\epsilon B,B)$ for sufficiently small~$\epsilon$, and for
$(A,B)\in L^{(i)}$ with $i=1,2$ this follows from \cite[Lemma 6.2.2]{SikThesis}. The second
part of the first assertion is immediate from our choices of the
$f_{I,J}$'s.  The second assertion follows from the fact that the
height of $(A,B)$ is a homogeneous function of degree $24$ in the
coefficients of $A$ and $B$.

Let $\FF$ denote a fundamental domain in $G_\R$ for the left action of
$G_\Z$ on $G_\R$ that is contained in a standard Siegel
set~\cite[\S2]{BH}. We may assume that $\FF=\{nak:n\in N'(a),a\in A',k\in K\}$, where
\begin{eqnarray*}\label{eqiwasawa}
  &K&=\{{\rm subgroup\; of \;orthogonal\; transformations\; } 
\SO_2(\R)\times\SO_4(\R)\subset G_\R\};\\
  &A'&=\{a(s_1,s_2,s_3,s_4):s_1>c_1;s_2,s_3,s_4>c_2\},\\
  &&\;\;\;\;\;\;\;\;{\rm where}\;a(s_1,s_2,s_3,s_4)=
\left[\left(\begin{array}{cc}
{s_1^{-1}} & {}\\ {} & {s_1}
\end{array}\right),
\left(\begin{array}{cccc}
    {s_2^{-3}s_3^{-1}s_4^{-1}} & {}&{}&{}\\ {} & \!\!\!\!\!\!\!\!\!\!\!\!\!\!\!\!\!{s_2s_3^{-1}s_4^{-1}}&{}&{}\\{} & {}&\!\!\!\!\!\!\!\!\!\!\!\!\!{s_2s_3s_4^{-1}}&{}\\{} & {}&{}&\!\!\!\!\!\!\!\!\!\!{s_2s_3s_4^{3}}
\end{array}\right)\right];\\
&N'&=\{n(u_1,\cdots,u_7):(u_i)\in\nu(a)\},\;\\
&&\;\;\;\;\;\;\;\;{\rm where}\; n(u)=\left[\left(\begin{array}{cc}
{1} & {} \\ {u_1} & {1}
\end{array}\right),
\left(\begin{array}{cccc}
{1} & {}&{}&{}\\ {u_2} & {1}&{}&{}\\{u_3}&{u_4}&{1}&{}\\{u_5} & {u_6}&{u_7}&{1}
\end{array}\right)\right];
\end{eqnarray*}
here $\nu(a)$ is a bounded and measurable subset of $[-1/2,1/2]^7$
depending only on $a\in A'$, and $c_1,c_2>0$ are absolute constants.

Fix $i\in\{0\#,1,2\}$. For $h\in G_\R$, we regard $\FF h\cdot R^{(i)}$
as a multiset, where the multiplicity of an element $v\in V_\R$ is
equal to $\#\{g\in\FF:v\in gh\cdot R^{(i)}\}$.  As in~\cite[\S2.1]{BS}, it follows that for any $h\in G_\R$ and any
$v\in V_\R^{(i)}$, the $G_\Z$-orbit of $v$ is represented $m(v)$ times in $\FF h\cdot R^{(i)}$, where
$$
m(v):=\#\Stab_{G_\R}(v)/\#\Stab_{G_\Z}(v).
$$ That is, the sum of the multiplicity in $\FF h\cdot R^{(i)}$ of $v'$,
over all $v'$ that are $G_\Z$-equivalent to $v$, is equal to $m(v)$.

The set of elements in $V_\R^{(i)}$ that have a nontrivial stabilizer
in $G_\Z$ has measure $0$ in $V_\R^{(i)}$. Thus, by
Proposition~\ref{propgrvr}, for any $h\in G_\R$ the multiset $\FF
h\cdot R^{(i)}$ is an $n_i$-fold cover of a fundamental domain for the
action of $G_\Z$ on $V_\R^{(i)}$, where $n_1=4$ and $n_{0\#}=n_2=8$.

It follows that if we let $R^{(i)}(X)$ denote the set of
elements in $R^{(i)}$ having height bounded by $X$, then for any $G_\Z$-invariant set $S\subset
V_\Z$, the product
$n_iN(S^\irr;X)$ is equal to the number of elements
in $\FF g\cdot R^{(i)}(X)\cap S^{\irr}$, with the slight caveat that the
(relatively rare---see Proposition \ref{propstrongnottot}) elements
with $G_\Z$-stabilizers of size $r$ ($r>1$) are counted with weight
$1/r$.

Counting strongly irreducible integer points in a single such region $\FF
g\cdot R^{(i)}(X)$ is difficult because it is an unbounded region. As in
\cite{BS}, we simplify the counting by suitably averaging over a continuous
range of elements $g$ lying in a compact subset of $G_\R$.

\subsection{Averaging and cutting off the cusp}\label{s32}

Throughout this section, we let $dg$ denote the Haar measure on $G_\R$
obtained from its Iwasawa decomposition $G_\R=NAK$ normalized in the
following way: for $g=nak$ with $n=n(u_1,\ldots u_7)\in N$,
$a=a(s_1,\ldots,s_4)\in A$, and $k\in K$,
we set
$$dg=s_1^{-2}s_2^{-12}s_3^{-8}s_4^{-12}
\prod_idu_i\,d^\times s_1\,d^\times s_2\,d^\times s_3\,d^\times s_4\,dk,$$
where $d^\times s$ denotes $s^{-1}ds$ and $dk$ is Haar measure on $K$
normalized so that $\int_Kdk=1$.

Let $G_0\subset G_\R$ be a compact, semialgebraic, left $K$-invariant subset that is the
closure of some nonempty open set in $G_\R$.  
Fix $i$ to be equal to
$0\#$, $1$, or $2$. Then, by the arguments of \S\ref{s31},
we may write

\begin{equation}\label{eqbeforeavg}
N(S;X)=\frac{\int_{g\in G_0}\#\{\FF g\cdot R^{(i)}(X)\cap S^\irr\}dg\;}{C_{G_0}},
\end{equation}
where $C_{G_0}=n_i\int_{g\in G_0}dg$. We use the right hand side of
\eqref{eqbeforeavg} to define $N(S;X)$ also for sets $S\subset V_\Z$ that
are not necessarily $G_\Z$-invariant.  

Identically as in \cite[Theorem~2.5]{BS}, the right hand side of (\ref{eqbeforeavg}) is equal to
\begin{equation}\label{eqavgimp}
\frac1{C_{G_0}}\int_{g\in N'(a)A'}\#\{S^\irr\cap 
  B(n,a;X)\}s_1^{-2}s_2^{-12}s_3^{-8}s_4^{-12}du\, d^\times s\,
\end{equation}
where $B(n,a;X):=na G_0\cdot R^{(i)}(X)$ and $d^\times s:=d^\times \!s_1\,d^\times \!s_2\,d^\times \!s_3\,d^\times \!s_4$.

To estimate the number of integer points in the bounded multiset
$B(n,a;X)$, we use the following proposition due to Davenport.
\begin{proposition}[\cite{Davenport1}]\label{propdavenport}
    Let $\mathcal R$ be a bounded, semi-algebraic multiset in $\R^n$
  having maximum multiplicity $m$, and that is defined by at most $k$
  polynomial inequalities each having degree at most $\ell$.  
  Then the number of integer lattice points $($counted with
  multiplicity$)$ contained in the region $\mathcal R$ is
\[\Vol(\mathcal R)+ O(\max\{\Vol(\bar{\mathcal R}),1\}),\]
where $\Vol(\bar{\mathcal R})$ denotes the greatest $d$-dimensional 
volume of any projection of $\mathcal R$ onto a coordinate subspace
obtained by equating $n-d$ coordinates to zero, where 
$d$ takes all values from
$1$ to $n-1$.  The implied constant in the second summand depends
only on $n$, $m$, $k$, and $\ell$.
\end{proposition}

Proposition \ref{propdavenport} yields a good estimate on the number
of integer points in $B(n,a;X)$ when the~$s_i$'s
($a=a(s_1,s_2,s_3,s_4)$) are bounded by a small power of $X$ (we shall
make this more precise in what follows). Our next aim is to show that
when one of the $s_i$'s is large relative to $X$, the set $B(n,a;X)$ has
very few strongly irreducible integer points.  To this end, we first
give conditions that guarantee that an element in $V_\Z$ is not
strongly irreducible.

\begin{lemma}\label{lemcondred}
  Let $(A,B)$ be a point in $V_\Z$ expressed in the form $(\ref{eqAB})$, and 
suppose that one of the following four conditions is satisfied:
  \begin{itemize}
  \item[{\rm (1)}] $a_{11}=a_{12}=a_{13}=a_{14}=0;$
  \item[{\rm (2)}] $a_{11}=a_{12}=a_{13}=a_{22}=a_{23}=0;$
  \item[{\rm (3)}] $a_{11}=a_{12}=a_{13}=b_{11}=b_{12}=b_{13}=0;$
  \item[{\rm (4)}] $a_{11}=a_{12}=a_{22}=b_{11}=b_{12}=b_{22}=0.$
  \end{itemize}
  Then $(A,B)$ is not strongly irreducible.
\end{lemma}
\begin{proof}
  In the first two cases, we see that $\det(A)=0$. This implies that the
  $x^4$-coefficient of $f(x,y)$ is equal to zero; hence $f(x,y)$ is
  reducible over $\Q$ and $(A,B)$ is not strongly irreducible.

  In the last two cases, the binary quartic resolvent
  form $f(x,y)$ of $(A,B)$ has a multiple root over $\bar\Q$. Thus,
  the discriminant $\Delta(A,B)=\Delta(f)$ of $(A,B)$ is equal to zero and
  so again $(A,B)$ is not strongly irreducible.
\end{proof}

Next, note that the action of
$a(s_1,s_2,s_3,s_4)$ on $(A,B)\in V_\R$ scales each coordinate $t$
($=a_{ij}$~or~$b_{ij}$) of $V_\R$ by a rational
function $w(t)$ in the $s_i$'s.  
We define the {\it weight} of a product of such coordinates to be the
product of the weights of these coordinates. Then evidently the size of the
coordinate $t$ of an element in $B(n,a;X)$ is 
$O(X^{1/24}w(t))$. For example, we have
$w(a_{11})=s_1^{-1}s_2^{-6}s_3^{-2}s_4^{-2}$, and so if
$(A,B)=((a_{ij}),(b_{ij}))\in B(n,a;X)$, then
$a_{11}=O(X^{1/24}w(a_{11}))$.  We now have the following lemma:
\begin{lemma}\label{lemsibound}
  Let $na(s_1,s_2,s_3,s_4)\in N'(a)A'$ be such that $V_\Z^\irr\cap
  B(n,a;X)$ is nonempty. Then $s_i=O(X^{1/24})$ for
  $i\in\{1,\ldots,4\}$.
\end{lemma}
\begin{proof}
  Let $na$ be an element satisfying the hypothesis of the lemma. Since
  $B(n,a;X)$ contains an integral point $(A,B)$ not satisfying any of
  the four conditions of Lemma \ref{lemcondred}, we see that
  $X^{1/24}w(t)\gg 1$ for $t=a_{14}$, $a_{23}$, $b_{13}$, and
  $b_{22}$. Thus, we obtain the following four estimates:
$$
{\rm (1)}\; \frac{s_1s_2^2}{s_4^{2}}=O(X^{1/24}),\quad
{\rm (2)}\; \frac{s_1s_4^{2}}{s_2^2}=O(X^{1/24}),\quad
{\rm (3)}\; \frac{s_2^2s_4^2}{s_1}=O(X^{1/24}),\quad
{\rm (4)}\; \frac{s_3^2s_4^2}{s_1s_2^2}=O(X^{1/24}).
$$
Multiplying the first two estimates immediately yields
$s_1=O(X^{1/24})$. Using this bound on $s_1$ and the third estimate then
gives $s_2=O(X^{1/24})$ and $s_4=O(X^{1/24})$. Finally, multiplying the first and
fourth estimates gives $s_3=O(X^{1/24})$, completing the proof of the lemma.
\end{proof}

We now prove the following estimate which bounds the number of
strongly irreducible points in $\FF g\cdot R^{(i)}(X)\cap V_\Z$
that have $a_{11}=0$, as we average over $g\in G_0$. More precisely:
\begin{lemma}\label{lemtable}
We have
$$\int_{g\in N'(a)A'}\#\{(A,B)\in V_\Z^\irr\cap 
  B(n,a;X):a_{11}=0\}s_1^{-2}s_2^{-12}s_3^{-8}s_4^{-12}du\, d^\times s=O(X^{19/24}).$$
\end{lemma}
\begin{proof}
  The proof of this lemma is very similar to that of \cite[Lemma
  11]{dodpf}. 
  We partition the set $\{(A,B)\in V_\Z^\irr:a_{11}=0\}$ into fourteen
  subsets defined by setting certain coordinates of $(A,B)\in
  V_\Z^\irr$ equal to zero and certain other coordinates to be
  nonzero. These sets are listed in the second column of
  Table~\ref{table1}, and it follows from Lemma~\ref{lemcondred} that they do
  indeed form a partition.

  For any subset $T\subset V_\Z$, let us define $N^*(T,X)$ by
$$N^*(T,X):=\int_{\substack{g\in N'(a)A'\\s_i=O(X^{1/24})}}\#\{T\cap B(n,a;X)\}dg.$$
Then Lemma \ref{lemsibound}, together with the bound $N^*(T,X)=O(X^{19/24})$
for the fourteen sets $T$ listed in Table~\ref{table1}, imply Lemma
\ref{lemtable}.

\begin{table}[t]
  \centering
  \begin{tabular}{|c | c| c| c|}
    \hline
    Case & The set $T\subset V_\Z^\irr$ defined by & $N^*(T,X)\ll$ & Use factor\\
    \hline
    \hline
    1&$a_{11}=0$&$X^{19/24}$&-\\
    &$a_{12},b_{11}\neq 0$&$$&\\
    \hline
    2a&$a_{11},a_{12}=0$&$X^{18/24+\epsilon}$&-\\
    &$a_{13},a_{22},b_{11}\neq 0$&&\\
    \hline
    2b&$a_{11},b_{11}=0$&$X^{18/24+\epsilon}$&-\\
    &$a_{12}\neq 0$&&\\
    \hline
    3a&$a_{11},a_{12},a_{13}=0$&$X^{18/24+\epsilon}$&$a_{14}$\\
    &$a_{14},a_{22},b_{11}\neq 0$&&\\
    \hline
    3b&$a_{11},a_{12},a_{22}=0$&$X^{18/24+\epsilon}$&$a_{13}$\\
    &$a_{13},b_{11}\neq 0$&&\\
    \hline
    3c&$a_{11},a_{12},b_{11}=0$&$X^{18/24+\epsilon}$&$a_{13}$\\
    &$a_{13},a_{22},b_{12}\neq 0$&&\\
    \hline
    4a&$a_{11},a_{12},a_{13},a_{22}=0$&$X^{18/24+\epsilon}$&$a_{14}^2$\\
    &$a_{14},a_{23},b_{11}\neq 0$&&\\
    \hline
    4b&$a_{11},a_{12},a_{13},b_{11}=0$&$X^{18/24+\epsilon}$&$a_{14}b_{12}$\\
    &$a_{14},a_{22},b_{12}\neq 0$&&\\
    \hline
    4c&$a_{11},a_{12},a_{22},b_{11}=0$&$X^{16/24+\epsilon}$&-\\
    &$a_{13},b_{12}\neq 0$&&\\
    \hline
    4d&$a_{11},a_{12},b_{11},b_{12}=0$&$X^{18/24+\epsilon}$&$a_{13}^{2}$\\
    &$a_{13},a_{22}\neq 0$&&\\
    \hline
    5a&$a_{11},a_{12},a_{13},a_{22},b_{11}=0$&$X^{16/24+\epsilon}$&$a_{14}$\\
    &$a_{14},a_{23},b_{12}\neq 0$&&\\
    \hline
    5b&$a_{11},a_{12},a_{13},b_{11},b_{12}=0$&$X^{18/24+\epsilon}$&$a_{14}^2b_{13}$\\
    &$a_{14},a_{22},b_{13}\neq 0$&&\\
    \hline
    5c&$a_{11},a_{12},a_{22},b_{11},b_{12}=0$&$X^{18/24+\epsilon}$&$a_{13}^2b_{22}$\\
    &$a_{13},b_{22}\neq 0$&&\\
    \hline
    6&$a_{11},a_{12},a_{13},a_{22},b_{11},b_{12}=0$&$X^{18/24+\epsilon}$&$a_{14}^{2}b_{13}b_{22}$\\
    &$a_{14},a_{23},b_{13},b_{22}\neq 0$&&\\
    \hline
  \end{tabular}
  \caption{Estimates on the number of strongly irreducible points in cuspidal regions}\label{table1}
\end{table}

We now describe how the required bound on $N^*(T,X)$ may be obtained
for Cases 1, 2a, and 3a of Table~\ref{table1}. In
Case 1, we have
\begin{eqnarray*}
N^*(T,X)&=&O\Bigl(\int_{\substack{g\in N'(a)A'\\s_i=O(X^{1/24})}}\frac{X^{20/24}}{X^{1/24}w(a_{11})}s_1^{-2}s_2^{-12}s_3^{-8}s_4^{-12}du\, d^\times s\Bigr)\\[0.1in]&=&O\Bigl(X^{19/24}\int_{\substack{g\in N'(a)A'\\s_i=O(X^{1/24})}}s_1^{-1}s_2^{-6}s_3^{-6}s_4^{-10}du\,d^\times s\Bigr).
\end{eqnarray*}
Since the $s_i$'s are bounded from below, we obtain the required
bound. 

Similarly, in Case 2a, we have
\begin{eqnarray*}
N^*(T,X)&=&O\Bigl(\int_{\substack{g\in N'(a)A'\\s_i=O(X^{1/24})}}\frac{X^{20/24}}{X^{2/24}w(a_{11})w(a_{12})}s_1^{-2}s_2^{-12}s_3^{-8}s_4^{-12}du\, d^\times s\Bigr)\\[0.1in]&=&
O\Bigl(X^{18/24}\int_{\substack{g\in N'(a)A'\\s_i=O(X^{1/24})}}s_2^{-4}s_3^{-4}s_4^{-8}du\,d^\times s\Bigr).
\end{eqnarray*}
Again, since the $s_i$ are
bounded by $O(X^{1/24})$, we obtain the required bound
$N^*(T,X)=O(X^{18/24+\epsilon})$.

Finally, in Case 3a, we have
\begin{eqnarray*}
N^*(T,X)&=&O\Bigl(\int_{\substack{g\in N'(a)A'\\s_i=O(X^{1/24})}}\frac{X^{20/24}}{X^{3/24}w(a_{11})w(a_{12})w(a_{13})}s_1^{-2}s_2^{-12}s_3^{-8}s_4^{-12}du\, d^\times s\Bigr)\\[0.1in]&=&
O\Bigl(\int_{\substack{g\in N'(a)A'\\s_i=O(X^{1/24})}}\frac{X^{20/24}\cdot X^{1/24}w(a_{14})}{X^{3/24}w(a_{11})w(a_{12})w(a_{13})}s_1^{-2}s_2^{-12}s_3^{-8}s_4^{-12}du\, d^\times s\Bigr)\\[0.1in]&=&
O\Bigl(X^{18/24}\int_{\substack{g\in N'(a)A'\\s_i=O(X^{1/24})}}s_2^{-4}s_3^{-4}s_4^{-4}du\,d^\times s\Bigr),
\end{eqnarray*}
where the second equality follows by multiplying the integrand by
$X^{1/24}$ times the weight of the ``factor'' listed in the fourth
column of Table~\ref{table1}.  This yields an upper bound for the desired integral
since the factor (which is an integer) 
was assumed to be nonzero, and therefore is at least 1 in absolute
value; hence the corresponding weight must also be bounded from below
by an absolute positive constant. As in Case 2a, we obtain the
required bound $N^*(T,X)=O(X^{18/24+\epsilon})$.

The proof of the bound for the other eleven cases are
identical. This concludes the proof of Lemma \ref{lemtable}.
\end{proof}

We have proven that the number of irreducible elements in the
``cuspidal regions'' of the fundamental domain is negligible. The next
lemma states that the number of reducible elements in the ``main
body'' of the fundamental domain is negligible:

\begin{lemma}\label{lemfewred}
  With notation as above, we have
  \begin{eqnarray*}
\int_{g\in N'(a)A'}\#\{(A,B)\in V_\Z^\red\cap 
B(n,a;X):a_{11}\neq 0\}dg=o(X^{5/6}),     
  \end{eqnarray*}
where $V_\Z^\red$ denotes the set of
elements in $V_\Z$ that are not strongly reducible.
\end{lemma}

Therefore, in order to estimate $N(V_\Z;X)$, it suffices to count the total
number of (not necessarily strongly irreducible) integral points in
the main body of the fundamental domain. We do this in the following
proposition:

\begin{proposition}\label{propequalsvol}
With notation as above, we have
\begin{equation*}
\frac1{C_{G_0}}\int_{g\in N'(a)A'}\#\{(A,B)\in V_\Z\cap 
B(n,a;X):a_{11}\neq 0\}dg=
\frac1{n_i}\Vol(\FF\cdot R^{(i)}(X))+o(X^{5/6}),
\end{equation*}
where the volume of
sets in $V_\R$ is computed with respect to Euclidean measure
normalized so that $V_\Z$ has covolume $1$.
\end{proposition}

\begin{proof}
  The proof of Proposition \ref{propequalsvol} is similar to that
  of \cite[Proposition 12]{dodpf}. If $v\in B(n,a;X)$, then we
  know that $a_{12}(v)=O(X^{1/60}w(a_{12}))$. Thus, from Propositions
  \ref{lemtable} and \ref{lemfewred}, we obtain
\begin{equation}\label{proofeq}
\displaystyle\frac{1}{C_{G_0}}\int_{na\in\FF}\#\{B(n,a;X)\cap V^\irr_\Z\}dnd^\times a=\displaystyle\frac{1}{C_{G_0}}\int_{\substack{na\in\FF\\X^{1/24}w(a_{11})\gg 1}}\#\{B(n,a;X)\cap V_\Z\}dnd^\times a+o(X^{5/6}).
\end{equation}
Since $a_{11}$ has minimal weight, and the projection of
$B(n,a;X)$ onto the $a_{11}$-axis has length greater than an absolute
positive constant when $X^{1/24}w(a_{11})\gg 1$, Proposition
\ref{propdavenport} implies that the main term on the right hand side of
\eqref{proofeq} is equal to
\begin{equation}\label{eqproofsmallvol}
\displaystyle\frac{1}{C_{G_0}}\int_{\substack{na\in\FF\\X^{1/24}w(a_{11})\gg 1}}\Vol(B(n,a;X))+O\Bigl(\frac{\Vol(B(n,a;X))}{X^{1/24}w(a_{11})}\Bigr)dnd^\times a.
\end{equation}
Since the region $\{nak\in\FF:w(a_{11})\ll X^\epsilon\}$ has volume
$o(1)$ for any fixed $\epsilon$, \eqref{eqproofsmallvol} is equal to
$$
\displaystyle\frac{1}{C_{G_0}}\int_{na\in\FF}\Vol(B(n,a;X))dnd^\times a+o(X^{5/6}).
$$
The proposition follows since
$$
\displaystyle\frac{1}{C_{G_0}}\int_{na\in\FF}\Vol(B(n,a;X))dnd^\times a=\displaystyle\frac{1}{C_{G_0}}\int_{h\in G_0}\Vol(\FF h\cdot R^\pm(X))dh,
$$
and the volume of $\FF h\cdot R^\pm(X)$ is independent of $h$.
\end{proof}

Lemmas \ref{lemtable} and \ref{lemfewred} and Proposition~\ref{propequalsvol}
imply
that, up to an error of $o(X^{5/6})$, the quantity $n_i\cdot
N(V_\Z^{(i)};X)$ is equal to the volume of $\FF\cdot
R^{(i)}(X)$ for $i=0\#$, $1$, and $2$. In the next section, we
obtain a useful expression for this volume.

\subsection{Computing the volume}\label{s33}

Recall that at the beginning of Section 3, we fixed an algebraic
nonzero top-degree left invariant differential form $\omega$ on $G$
such that for all primes $p$, the measure of $G_{\Z_p}$ with respect
to $\omega$ is $\#G_{\F_p}/p^{18}$. Let $dv$ denote Euclidean measure
on $V_\R$ normalized so that $V_\Z$ has covolume $1$. Finally, note
that for $i=0\#$, $1$, and $2$, the sets $R^{(i)}$ contain at most one
point $p_{I,J}$ having invariants $I$ and $J$ for any pair
$(I,J)\in\R\times\R$. We choose $dIdJ$ to be the measure on $R^{(i)}$.

With these measure normalizations, we have the following proposition
whose proof is identical to that of \cite[Proposition 2.8]{BS}:
\begin{proposition}\label{bqjac}
For any measurable function $\phi$
  on $V_\R$, we have
\begin{equation}\label{Jac}
|\J|\cdot\int_{p_{I,J}\in R^{(i)}}
\int_{h\in G_\R}\phi(h\cdot p_{I,J}))\,\omega(h)\,dI dJ=\int_{G_\R\cdot R^{(i)}}\phi(v)dv=n_i\int_{\pnr}\phi(v)dv,
\end{equation}
where $\J$ is a nonzero constant in $\Q$ independent of $i$.
\end{proposition}

Using Proposition \ref{bqjac}, it is easy to compute the volume of the
multiset $\FF\cdot R^{(i)}(X)$:

\begin{equation}
  \int_{\FF\cdot R^{(i)}(X)}\!\!\!\!\!dv=
  |\J|\cdot\int_{p_{I,J}\in R^{(i)}(X)}\int_{\FF}\omega(h)\,dI\,dJ=|\J|\cdot\Vol(\FF)\int_{R^{(i)}(X)}dI\,dJ.
\end{equation}
Up to an error of $O(X^{1/2})$, the quantity $\int_{R^{(i)}(X)}dI\,dJ$ is
equal to $N^{+}(X)$ when $i=0\#$ or $2$, and $N^{-}(X)$
when $i=1$ (see the proof of \cite[Proposition 2.10]{BS} for details).

We conclude that
\begin{equation}
\begin{array}{ccc}
N(V_{\Z}^{(1)};X)&=&\displaystyle\frac14|\J|\cdot\Vol(G_\Z\backslash G_\R)N^-(X)+o(X^{5/6}),\\[0.1in]
N(V_{\Z}^{(i)};X)&=&\displaystyle\frac1{8}|\J|\cdot\Vol(G_\Z\backslash G_\R)N^+(X)+o(X^{5/6}),
\end{array}
\end{equation}
for $i=0\#$ and $2$.
We thus obtain Theorem \ref{thsec2main}.

\subsection{Congruence conditions}\label{s34}
In this subsection, we present a version of Theorem \ref{thsec2main} in which we
count pairs of integral quaternary quadratic forms satisfying any finite set of
congruence conditions.

For any set $S$ in $V_\Z$ that is
definable by congruence conditions, let us denote by $\mu_p(S)$
the $p$-adic density of the $p$-adic closure of $S$ in $V_{\Z_p}$,
where we normalize the additive measure $\mu_p$ on $V_{\Z_p}$ so that
$\mu_p(V_{\Z_p})=1$.
We then have the following theorem whose proof is identical to that of \cite[Theorem 2.11]{BS}.
\begin{theorem}\label{cong2}
Suppose $S$ is a subset of $\pnv$ defined by congruence conditions modulo finitely many prime powers. Then we have 
\begin{equation}\label{ramanujan}
N(S\cap\pnv;X)
  = N(\pnv;X)
  \prod_{p} \mu_p(S)+o(X^{5/6}),
\end{equation}
where $\mu_p(S)$ denotes the $p$-adic density of $S$ in $V_\Z$, and
where the implied constant in $o(X^{5/6})$ depends only on $S$.
\end{theorem}
We furthermore have the following weighted version of Theorem \ref{cong2} whose proof is identical to that of \cite[Theorem 2.12]{BS}.

\begin{theorem}\label{cong3}
  Let $p_1\ldots,p_k$ be distinct prime numbers. For $j=1,\ldots,k$,
  let $\phi_{p_j}:V_\Z\to\R$ be bounded $G_\Z$-invariant functions on
  $V_\Z$ such that $\phi_{p_j}(v)$ depends only on the congruence
  class of $v$ modulo some power $p_j^{a_j}$ of $p_j$.  Let
  $N_\phi(V_\Z^{(i)};X)$ denote the number of strongly irreducible
  $G_\Z$-orbits in $V_\Z^{(i)}$ having height bounded by $X$,
  where each orbit $G_\Z\cdot v$ is counted with weight
  $\phi(v):=\prod_{j=1}^k\phi_{p_j}(v)$. Then we have
\begin{equation}
N_\phi(V_\Z^{(i)};X)
  = N(V_\Z^{(i)};X)
  \prod_{j=1}^k \int_{v\in V_{\Z_{p_j}}}\tilde{\phi}_{p_j}(v)\,dx+o(X^{5/6}),
\end{equation}
where $\tilde{\phi}_{p_j}$ is the natural extension of ${\phi}_{p_j}$
to $V_{\Z_{p_j}}$ by continuity, $dv$ denotes the additive
measure on $V_{\Z_{p_j}}$ normalized so that $\int_{v\in
  V_{\Z_{p_j}}}dv=1$, and where the implied constant in the error term
depends only on the local weight functions ${\phi}_{p_j}$.
\end{theorem}

\subsection[The number of reducible points in the main bodies of the
fundamental domains is negligible]{The number of reducible points and points with large stabilizers in
  the main bodies of the fundamental domains is negligible}\label{s35} 

In this section we prove Lemma \ref{lemfewred}, which states that the number of integral elements that are not strongly irreducible in the main body of the fundamental domain is negligible. We then prove that the number of strongly irreducible $G_\Z$-orbits on elements with a nontrivial stabilizer in $G_\Q$ having bounded height is negligible.

\vspace{0.1in}
\noindent{\bf Proof of Lemma \ref{lemfewred}:}
An element $(A,B)\in V_\Z$ with $\Delta(A,B)\neq 0$ fails to be strongly irreducible if and only if the binary quartic resolvent form $f_{A,B}(x,y)=16\det(Ax+By)$ has a root in $\P^1_{\Q}$. Let $p>3$ be prime. If $f(x,y)$ has a root in
$\P^1_{\Q}$, then the reduction of $f(x,y)$ modulo
$p$ has a root in $\P^1_{\F_p}$. We construct elements $(A,B)\in
V_{\F_p}$, for a positive density family of primes $p$, such that
$f_{A,B}(x,y)$ has no root in $\P^1(\F_p)$.

Let $p$ be a prime congruent to $3$ modulo $4$ such that there exists an element $s\in\F_p$ satisfying $s^2=-2$.
Consider the pair
\begin{equation}\label{eqAB2}
(A,B)=\left( \left[ \begin{array}{cccc} 1 &  &  & \\  & 1 &  & \\  &  & 1 & \\  &  &  &1 \end{array} \right],
  \left[ \begin{array}{cccc}  & 1 &  & \\ 1 &  & s & \\  & s &  & 1\\  &  & 1 & \end{array} \right]\right).
\end{equation}
We have $\det(Ax+By)=x^4+y^4$, implying that $f_{A,B}(x,y)$ has no
root defined over $\F_p$. 
Therefore, if the reduction modulo $p$ of $(A,B)\in V_\Z$ is
$G_{\F_p}$-equivalent to any $\F_p^\times$-multiple of the right hand
side of (\ref{eqAB2}), then $(A,B)$ is not strongly
irreducible. 
Since $\#\{g\cdot\lambda\cdot (A,B):g\in
G_{\F_p},\;\lambda\in\F_p^\times\}\gg \#V_{\F_p}/p$, we obtain
\begin{equation*}
  \int_{g\in N'(a)A'}\#\{(A,B)\in V_\Z^\red\cap B(n,a;X):a_{11}\neq 0\}dg=O\Bigl(X^{5/6}\prod_{\substack{p\equiv 3\!\!\!\!\pmod{4}\\p<Y}}(1-p^{-1})\Bigr)
\end{equation*}
for any $Y>0$. Letting $Y\to\infty$ yields
Lemma \ref{lemfewred}. $\Box$

\begin{proposition}\label{propstrongnottot}
  The number of $G_\Z$-orbits on elements on $V_\Z$ that are strongly
  irreducible, have height bounded by $X$, and have a nontrivial
  stabilizer in $G_\Q$ is $o(X^{5/6})$.
\end{proposition}
\begin{proof}
  Proposition \ref{propparamfield} implies that an element $(A,B)\in
  V_\Z$ having invariants $I$ and $J$ has a nontrivial stabilizer in
  $G_\Q$ if and only if the elliptic curve
  $E:y^2=g_{A,B}(x)=x^3-\frac{I}{3}-\frac{J}{27}$ contains a nontrivial
  $4$-torsion point over $\Q$, which happens exactly when $g(x)$ has a
  rational root.

Let $p$ be a prime congruent to $1$ modulo $3$. Let $t\in\F_p$ be an
element having no solution $a^3=t$ for $a\in\F_p$.  Consider the pair $(A,B)$ given by
\begin{equation}\label{eqAB1}
2(A,B)=\left( \left[ \begin{array}{cccc} 0 &  &  & \\  &  &  & 1\\  &  & 1 & \\  & 1 &  & \end{array} \right],
\left[ \begin{array}{cccc} -1 &  &  & \\  &  & 1 & \\  & 1 &  & \\  &  &  & -t\end{array} \right]\right).
\end{equation}
We have $16\det(Ax+By)=x^3y-ty^4$, implying that $g_{A,B}(x,y)=x^3-ty^3$ is
irreducible over $\F_p$.

Therefore, if the reduction modulo $p$ of $(A,B)\in V_\Z$ is
$G_{\F_p}$-equivalent to the right hand side of (\ref{eqAB1}) for any
prime $p$, then $(A,B)$ has a trivial stabilizer in $G_\Q$.
Proposition \ref{propstrongnottot} now follows from Lemma
\ref{lemtable} and an argument identical to the proof of Lemma
\ref{lemfewred}.
\end{proof}

\subsection{Tail estimates and a squarefree sieve}\label{s36}

In order to prove Theorems \ref{mainellip} and \ref{ellipcong}, we
require a stronger version of Theorem \ref{cong3}; one which counts
weighted $G_\Z$-orbits where the weights are defined by congruence
conditions modulo infinitely many prime powers. In this subsection, we
use the methods and results of \cite{geosieve} to prove the necessary result. 

We start with the following two definitions. 
A function $\phi:V_\Z\to[0,1]\subset\R$ is said to be {\it defined by congruence
  conditions} if, for all primes $p$, there exist functions
$\phi_p:V_{\Z_p}\to[0,1]$ satisfying the following conditions:
\begin{itemize}
\item[(1)] For all $(A,B)\in V_\Z$, the product $\prod_p\phi_p(A,B)$ converges to $\phi(A,B)$.
\item[(2)] For each prime $p$, the function $\phi_p$ is
locally constant outside some closed set $S_p \subset V_{\Z_p}$ of measure zero.
\end{itemize}
Such a function $\phi$ is called {\it acceptable} if, for sufficiently
large primes $p$, we have $\phi_{p}(A,B)=1$ whenever $p^2\nmid\Delta(A,B)$.

Then we will prove the following theorem:
\begin{theorem}\label{thsqsieve}
    Let $\phi:V_\Z\to[0,1]$ be an acceptable function that is defined by
  congruence conditions via the local functions $\phi_{p}:V_{\Z_p}\to[0,1]$. Then, with
  notation as in Theorem~$\ref{cong3}$, we have:
\begin{equation}
N_\phi(V_\Z^{(i)};X)
  = N(V_\Z^{(i)};X)
  \prod_{p} \int_{v\in V_{\Z_{p}}}\phi_{p}(v)\,dv+o(X^{5/6}).
\end{equation}
\end{theorem}

For a prime $p$, let $\W_p$ denote the set of elements in $V_\Z$ whose
discriminants are divisible by~$p^2$. The key ingredient needed to
prove Theorem \ref{thsqsieve} is the following tail estimate:
\begin{theorem}\label{thtailest}
  Let $\epsilon>0$ be fixed. Then for any $i\in\{0\#,1,2\}$ we have:
  \begin{equation}\label{equnif}
    N\bigl(\displaystyle\cup_{p>M}\W_p,X\bigr)=O_\epsilon(X^{5/6}/(M\log M)+X^{19/24})+O(\epsilon X^{5/6}).
  \end{equation}
%where the implied constant is independent of $M$ and $X$.
\end{theorem}
\begin{proof}
  Let $\W_p^{(1)}$ denote the set of elements in $(A,B)\in V_\Z$ whose
  discriminants are divisible by $p^2$ for (mod $p$) reasons; i.e.,
  $p^2$ divides the discriminant of $(A,B)+p(A',B')$ for every
  $(A',B')\in V_\Z$. For $\epsilon>0$, let
  $\FF^{(\epsilon)}\subset\FF$ denote the subset of elements
  $na(s_1,s_2,s_3,s_4)k\in\FF$ such that the $s_i$ are bounded
  above by an appropriate constant to ensure that
  $\Vol(\FF^{(\epsilon)})=(1-\epsilon)\Vol(\FF)$. Then
  $\FF^{(\epsilon)}\cdot R^{(i)}(X)$ is a bounded domain in
  $V_\R$ that expands homogeneously with $X$. By \cite[Theorem~3.3]{geosieve}, we obtain
\begin{equation}\label{unifest1}
\#\{\FF^{(\epsilon)}\cdot R^{(i)}(X)\bigcap (\cup_{p>M}\W^{(1)}_p)\}=O(X^{5/6}/(M\log M)+X^{19/24}).
\end{equation}
Also, the results of \S\ref{s31} and \S\ref{s32} imply that
\begin{equation}\label{unifest2}
\#\{(\FF\backslash\FF^{(\epsilon)})\cdot R^{(i)}(X)\bigcap V_\Z^\irr\}=O(\epsilon X^{5/6}).
\end{equation}
Combining the estimates \eqref{unifest1} and \eqref{unifest2} yields
\eqref{equnif} with $\W_p$ replaced with $\W^{(1)}_p$.

  Next, let $(A,B)$ be an element of
  $\W_p^{(2)}:=\W_p\backslash\W_p^{(1)}$ for some prime
  $p>2$. By definition, $v_p(\Delta(A,B))=v_p(\Delta(f))$, where
  $f=f_{A,B}$ is the binary quartic resolvent form of $(A,B)$. Thus $p^2$
  divides the discriminant of $f$, and since $(A,B)\notin\W_p^{(1)}$ we
  may assume that the reduction of $f$ modulo $p$ contains the square
  of a linear factor. By replacing $(A,B)$ with a $G_\Z$-translate, if
  necessary, we may further assume that $p^2$ divides the
  $x^4$-coefficient of $f(x,y)$ and $p$ divides the $x^3y$-coefficient
  of $f(x,y)$. This condition (along with the fact that
  $(A,B)\notin\W_p^{(1)}$) implies that we may assume
  $(A,B)=((a_{ij}),(b_{ij}))$ satisfies the following conditions: \pagebreak
  \begin{enumerate}
  \item $a_{12}\equiv a_{13}\equiv a_{14}\equiv b_{11}\equiv 0\pmod{p},$
  \item $a_{11}\equiv 0\pmod{p^2}.$
  \end{enumerate}
  If $\gamma_p\in G_\Q$ is defined by
  $$\gamma_p:=\left[\left(\begin{array}{cc}
{1} & {}\\ {} & {p}
\end{array}\right),
\left(\begin{array}{cccc}
    {p^{-1}} & {}&{}&{}\\ {} & {\!\!\!\!\!\!\!1}&{}&{}\\{} & {}&{1}&{}\\{} & {}&{}&{1}
\end{array}\right)\right],$$
then $\gamma_p\cdot(A,B)\in\W_p^{(1)}$ since it satisfies
$b_{22}\equiv b_{23}\equiv b_{24}\equiv b_{33}\equiv b_{34}\equiv
b_{44}\equiv 0\pmod{p}$. This yields a discriminant-preserving map
$\phi$ from $G_\Z$-orbits on $\W_p^{(2)}$ to $G_\Z$-orbits on
$\W_p^{(1)}$. The following lemma states that $\phi$ is at most $2$
to $1$:
\begin{lemma}
The map $\phi$ from $G_\Z$-orbits on $\W_p^{(2)}$ to
$G_\Z$-orbits on $\W_p^{(1)}$ is at most $2$ to~$1$.
\end{lemma}
\begin{proof}
  Let $(A,B)\in\W_p^{(1)}$ be any element in the image of $\phi$ and
  let $(\overline{A},\overline{B})$ denote its reduction modulo
  $p$. It is easy to see that $\phi^{-1}(A,B)$ is integral if and only
  if $b_{22}\equiv b_{23}\equiv b_{24}\equiv b_{33}\equiv b_{34}\equiv
  b_{44}\equiv 0\pmod{p}$. Therefore, the $G_\Z$-orbits on
  $\phi^{-1}(A,B)$ give rise to elements $[r:s]\in\P^1_{\F_p}$ along
  with a linear factor of the quadratic form corresponding to
  $r\overline{A}+s\overline{B}$. If there are two elements in
  $\P^1_{\F_p}$ such that the corresponding quadratic forms factor,
  then $(A,B)$ is $G_\Z$-equivalent to $(A_1,B_1)$, where the
  reductions of $A_1$ and $B_1$ modulo $p$ both factor over $\F_p$. We
  may thus assume that the bottom $3\times 3$ submatrix of $B_1$ is
  congruent to zero modulo $p$. If $(A_2,B_2)$ is
  $\gamma_p^{-1}(A_1,B_1)$, then we see that the reduction of $A_2$
  modulo $p$ also factors over $\F_p$, implying that
  $(A_2,B_2)\in\W_p^{(1)}$. Thus, $(A,B)$ cannot lie in the image of
  $\phi$ contradicting our hypothesis. Therefore, if $(A,B)$ is
  in the image of $\phi$, then there is exactly one element
  $[r:s]\in\P^1_{\F_p}$ such that the quadratic form corresponding to
  $r\overline{A}+s\overline{B}$ factors.

  We assume without loss of generality that $[r:s]=[0:1]$. If
  $\overline{B}$ has more than two linear factors, then
  $\overline{B}\equiv 0\pmod{p}$. Then it is easy to see that
  $\gamma_p^{-1}(A,B)\in\W_p^{(1)}$ because its binary quartic resolvent form
  is congruent to $0$ modulo $p$, again contradicting the hypothesis that $(A,B)$ is in the image of $\phi$. This concludes the proof of the
  lemma.
\end{proof}

Therefore, we obtain
\begin{equation}
  N\bigl(\displaystyle\cup_{p>M}\W^{(2)}_p(V),X\bigr)\leq 2N\bigl(\displaystyle\cup_{p>M}\W^{(1)}_p(V),X\bigr)=O_\epsilon(X^{5/6}/(M\log M)+X^{19/24})+O(\epsilon X^{5/6}),
\end{equation}
and Theorem \ref{thtailest} follows.
\end{proof}

Theorem \ref{thsqsieve} follows from Theorem \ref{thtailest} just as \cite[Theorem 2.21]{BS} followed from \cite[Theorem~2.13]{BS}.

\section{The average number of elements in the $4$-Selmer group of
  elliptic curves}

In this section, we prove Theorems \ref{mainellip} and \ref{ellipcong}
by computing the average size of the $4$-Selmer group of elliptic
curves over $\Q$, when these curves are ordered
by height.  In fact, we prove a generalization of these
theorems that allows us to average the size of the $4$-Selmer group of
elliptic curves whose defining equations satisfy certain {acceptable} sets
of local conditions. 

To state the theorem, we need the following
definitions. For any elliptic curve $E$ over $\Q$, we defined the
invariants $I(E)$ and $J(E)$ as in (\ref{eqEIJ}).  Let us denote the elliptic
curve having invariants $I$ and $J$ by $E^{I,J}$. 
Throughout his section we work with the slightly different height $H'$ on elliptic curves $E$, defined by
$$
H'(E):=H(I(E),J(E))=\max\{|I(E)^3|,J(E)^2/4\},
$$
so that the height on elliptic curves agrees with the height on $V_\Z$
defined in \eqref{eqABDH}. Note that since $H$ and $H'$ differ by a
constant factor of $4/27$, they induce the same ordering on the set of (isomorphism classes of) 
elliptic curves. 

For each prime $p$, let $\Sigma_p$ be a closed subset
of $\Z_p^2\backslash\{\Delta=0\}$ whose boundary has measure $0$. To
this collection $\Sigma=(\Sigma_p)_p$, we associate the family $F_\Sigma$ of elliptic
curves, where $E^{I,J}\in F_\Sigma$ if and only if $(I,J)\in \Sigma_p$
for all $p$. Such a family of elliptic curves over $\Q$ is said to be
{\it defined by congruence conditions}. We may also impose ``congruence
conditions at infinity'' on $F_\Sigma$ by insisting that an elliptic
curve $E^{I,J}$ belongs to $F_\Sigma$ if and only if $(I,J)$ belongs
to $\Sigma_\infty$, where $\Sigma_\infty$ is equal to
$\{(I,J)\in\R^2:\Delta(I,J)>0\}$, $\{(I,J)\in\R^2:\Delta(I,J)<0\}$, or
$\{(I,J)\in\R^2:\Delta(I,J)\neq 0\}$.

For such a family $F$ of elliptic curves defined by congruence conditions, let $\Inv(F)$ denote the set
$\{(I,J)\in\Z\times\Z:E^{I,J}\in F\}$, and let $\Inv_p(F)$ to be the
$p$-adic closure of $\Inv(F)$ in $\Z_p^2\backslash\{\Delta=0\}$.
Similarly, we define $\Inv_\infty(F)$ to be
$\{(I,J)\in\R^2:\Delta(I,J)>0\}$,
$\{(I,J)\in\R^2:\Delta(I,J)<0\}$, or $\{(I,J)\in\R^2:\Delta(I,J)\neq
0\}$ in accordance with whether $F$ contains only curves of positive
discriminant, negative discriminant, or both, respectively.
Such a family $F$ of elliptic curves is said to be {\it large} if, for
all but finitely many primes $p$, the set $\Inv_p(F)$ contains at
least those pairs $(I,J)\in\Z_p\times\Z_p$ such that $p^2\nmid
\Delta(I,J)$. 
Our purpose in this section is to prove the
following theorem which generalizes Theorems \ref{mainellip} and
\ref{ellipcong}.

\begin{theorem}\label{ellipall}
  Let $F$ be a large family of elliptic curves. When elliptic
  curves $E$ in $F$ are ordered by height, the average size of the
  $4$-Selmer group $S_4(E)$ is $7$.
\end{theorem}

\subsection{Computation of $p$-adic densities}\label{s41}
Throughout the rest of this section, we fix $F$ to be a large
family of elliptic curves. 
Proposition~\ref{propselparz} asserts that elements in the
$4$-Selmer group of the elliptic curve $E^{I,J}$ over $\Q$ are in
bijection with $G_\Q$-equivalence classes on the set of locally
soluble elements in $V_\Z$ having invariants $I$ and $J$.
Furthermore, elements of exact order $4$ in the $4$-Selmer group of
$E^{I,J}$ are in bijection with strongly irreducible
$G_\Q$-equivalence classes in the set of locally soluble elements in
$V_\Z$ having invariants $I$ and $J$.

In Section 2, we computed the asymptotic number of $G_\Z$-equivalence
classes of strongly irreducible elements in $V_\Z$ having bounded
height.  In order to use this to compute the number of
$G_\Q$-equivalence classes of strongly irreducible locally soluble
elements of $V_\Z$ having bounded height and invariants in $\Inv(F)$,
we count each strongly irreducible $G_\Z$-orbit $G_\Z\cdot x$
weighted by $\phi(x)$, where $\phi:V_\Z\to\R$ is a $G_\Z$-invariant
function that we now define.

For $x\in V_\Z$, let $B(x)$ denote a set of representatives for the
action of $G_\Z$ on the $G_\Q$-equivalence class of $x$ in $V_\Z$. We
define our weight function $\phi$ to be:
\begin{equation}\label{eqmx}
  \phi(x):=
  \begin{cases}
    \Bigl(\displaystyle\sum_{x'\in
      B(x)}\frac{\#\Aut_\Q(x')}{\#\Aut_\Z(x')}\Bigr)^{-1}
    \qquad&\text{if $x$ is locally soluble and $(I(x),J(x))\in
      \Inv_p(F)$ for all $p$;}\\[.1in]
    \qquad\qquad 0 \qquad&\text{otherwise},
  \end{cases}
\end{equation}
where $\Aut_\Q(x)$ and $\Aut_\Z(x)$ denote the stabilizers of $x\in
V_\Z$ in $G_\Q$ and $G_\Z$, respectively.

Note that if $x\in V_\Z$ has a trivial stabilizer in $G_\Q$, is locally soluble,
and satisfies $(I(x),J(x))\in \Inv(F)$, then $\phi(x)=\#B(x)^{-1}$. Thus,
Proposition~\ref{propstrongnottot} implies the following result:
\begin{proposition}\label{propprelimfinal}
  Let $F$ be a large family of elliptic curves. Following the notation
  of Theorem $\ref{cong3}$ and $\ref{thsqsieve}$, we have
$$
\sum_{\substack{E\in F\\H'(E)<X}}\#\{\sigma\in S_4(E):\sigma^2\neq 1\}
=N_{\phi}(V_\Z;X)+o(X^{5/6}).
$$
\end{proposition}

To evaluate the right hand side of the above equation using
Theorem \ref{thsqsieve}, we need to show that the weight function
$\phi$ is acceptable in the sense of Section 3.6. To this end, we define local
functions $\phi_p:V_{\Z_p}\to\R$ as follows.  For $x\in V_{\Z_p}$, let
$B_p(x)$ denote a set of representatives for the action of $G_{\Z_p}$
on the $G_{\Q_p}$-equivalence class of $x$ in $V_{\Z_p}$. Then we
define
\begin{equation}\label{eqmxp}
  \phi_p(x):=
  \begin{cases}
    \Bigl(\displaystyle\sum_{x'\in B_p(x)}\frac{\#\Aut_{\Q_p}(x')}{\#\Aut_{\Z_p}(x')}\Bigr)^{-1} &\text{if $x$ is soluble over $\Q_p$ and $(I(x),J(x))\in \Inv_p(F)$;}\\[.1in]
    \qquad 0 \qquad&\text{otherwise},
  \end{cases}
\end{equation}
where $\Aut_{\Q_p}(x)$ and $\Aut_{\Z_p}(x)$ denote the stabilizer of
$x\in V_{\Z_p}$ in $G_{\Q_p}$ and $G_{\Z_p}$, respectively.
Before we prove that $\phi$ is acceptable, we need the following lemma:
\begin{lemma}\label{lemsadek}
  For sufficiently large primes $p$, if $(A,B)\in V_{\Z_p}$ satisfies $\phi_p(A,B)\neq 1$, then the discriminant of $(A,B)$ is
  divisible by $p^2$.
\end{lemma}
\begin{proof}
  Since $F$ is a large family of elliptic curves, we know that
  for large enough primes $p$, if $(I,J):=(I(A,B),J(A,B))\notin
  \Inv_p(F)$, then $p^2\mid\Delta(A,B)$.

  Now suppose that $(I,J)\in \Inv_p(F)$ but $\phi_p(A,B)\neq
  1$. Then either $(A,B)$ is not soluble over~$\Q_p$,
  $\Aut_{\Q_p}(A,B)$ is not trivial, or $B_p(A,B)$ has size at least
  two. Let $C\in\P^3_{\F_p}$ be the curve cut out by the intersection
  of the quadrics defined by the reductions of $A$ and $B$ modulo
  $p$. The Lang-Weil estimates \cite{LW} imply that, for sufficiently
  large primes $p$, either $C$ is geometrically reducible or $C$
  has a smooth $\F_p$-point. Thus either
  $p^2$ divides the discriminant of $(A,B)$ or $(A,B)$ is locally
  soluble.

  Finally, \cite[Corollary 2.2]{Sadek1} implies that if $(A,B)$ is
  soluble and either $\Aut_{\Q_p}(A,B)$ is nontrivial or
  $\#B_p(A,B)>1$, then the reduction type of the elliptic curve
  $E^{I,J}$ over $\Q_p$ is not $I_0$ or $I_1$. This implies that
  $p^2\mid \Delta(E^{I,J})=\Delta(A,B)$,
  as desired.
\end{proof}

This leads us to the following proposition:
\begin{proposition}\label{propacceptable}
  The function $\phi:V_\Z^\irr\to\R$ is acceptable.
\end{proposition}
\begin{proof}
  The local weight functions $\phi_p$ are supported and locally constant outside the
  set of elements in $V_{\Z_p}$ having discriminant zero. That
  $\phi(A,B)=\prod_p\phi_p(A,B)$, for $(A,B)\in V_\Z$, follows from an
  argument identical to the proof of \cite[Proposition 3.6]{BS} and
  the fact that the class number of $G_\Q$ is $1$. Lemma~\ref{lemsadek} then
implies that $\phi$ is acceptable.
\end{proof}

We end the section with a proposition that evaluates $\int_{V_{\Z_p}}\phi_p(x)dx$.
\begin{proposition}\label{propmasseval}
  We have
  \begin{equation*}
    \begin{array}{rl}
    \displaystyle\int_{x\in V_{\Z_p}}\phi_p(x)dx&=\;\;|\J|_p\cdot\Vol(G_{\Z_p})\cdot\displaystyle\int_{(I,J)\in \Inv_p(F)}\displaystyle\frac{\#(E^{I,J}(\Q_p)/4E^{I,J}(\Q_p))}{\#(E^{I,J}(\Q_p)[4])}\\[.25in]&=\;\;
    \begin{cases}
      \phantom{55\cdot}\!|\J|_p\cdot\Vol(G_{\Z_p})\cdot\Vol(\Inv_p(F))\quad\text{if $p\neq 2$};\\[.05in]
      4\cdot|\J|_p\cdot\Vol(G_{\Z_p})\cdot\Vol(\Inv_p(F))\quad\text{if $p=2$},
    \end{cases}
  \end{array}
  \end{equation*}
where the volume of $\Inv_p(F)\subset\Z_p\times\Z_p$ is taken with
respect to the additive 
Haar measure on $\Z_p\times\Z_p$ normalized so that $\Vol(\Z_p\times\Z_p)=1$.
\end{proposition}
The first equality in Proposition~\ref{propmasseval} follows from an
argument identical to \cite[Proposition 3.9]{BS}. The second follows
from an argument identical to the proof of \cite[Lemma 3.1]{BK},
yielding that $\#(E^{I,J}(\Q_p)/4E^{I,J}(\Q_p))$ is equal to
$\#(E^{I,J}(\Q_p)[4])$ when $p\neq 2$ and equal to
$4\#(E^{I,J}(\Q_p)[4])$ when $p=2$.

\subsection{The proof of the main theorem (Theorem \ref{ellipall})}\label{s42}

We first state a theorem, proved in \cite[Theorem 3.17]{BS}, that counts
the number of elliptic curves having bounded height in a large family $F$.
\begin{theorem}\label{thellipcurvecount}
  Let $F$ be a large family of elliptic curves and let $N(F;X)$
  denote the number of elliptic curves in $F$ that have height bounded
  by $X$. Then
  \begin{equation}\label{eqnumelip}
    N(F;X)=\Vol(\Inv_\infty(F;X))\prod_p\Vol(\Inv_p(F))+o(X^{5/6}),
  \end{equation}
  where $\Inv_\infty(F;X)$ denotes the set of elements in
  $\Inv_\infty(F)$ that have height bounded by $X$.
\end{theorem}

For any large family $F$ of elliptic curves over $\Q$, 
it follows from Proposition \ref{propprelimfinal} that
\begin{equation}\label{4avgeval}
\displaystyle{\lim_{X\to\infty}\displaystyle\frac{\displaystyle
    \sum_{\substack{E\in
        F\\H'(E)<X}}\#\{\sigma\in S_4(E):\sigma^2\neq 1\}}{\displaystyle\sum_{\substack{E\in
        F\\H'(E)<X}}1}}=\lim_{X\to\infty}\displaystyle\frac{N_{\phi}(V_\Z;X)}{N(F;X)}.
\end{equation}
Proposition \ref{propacceptable} states that $\phi$ is
acceptable. Thus, the right hand side of (\ref{4avgeval})  
can be evaluated using Theorems~\ref{thsqsieve} and
\ref{thellipcurvecount}:
\begin{equation}\label{eqJcance}
\begin{array}{rcl}
\displaystyle\lim_{X\to\infty}\displaystyle\frac{N_{\phi}(V_\Z;X)}{N(F;X)}&=&\displaystyle\lim_{X\to\infty}\displaystyle\frac{\frac14|\J|\cdot\Vol(G_\Z\backslash G_\R)
\Vol(\Inv_\infty(F;X))\displaystyle\prod_p\int_{V_{\Z_p}}\phi_p(x)dx}{\Vol(\Inv_\infty(F;X))\displaystyle\prod_p\Vol(\Inv_p(F))}\\[.4in]
&=&\displaystyle\frac{|\J|\cdot\Vol(G_\Z\backslash G_\R)\displaystyle\prod_p\bigl(|\J|_p\cdot\Vol(G_{\Z_p})\cdot\Vol(\Inv_p(F))\bigr)}{\displaystyle\prod_p\Vol(\Inv_p(F))},
\end{array}
\end{equation}
where the second equality follows from Proposition \ref{propmasseval}. Since
$\Vol(G_{\Z_p})\prod_p\Vol(G_{\Z_p})$ is equal to the Tamagawa number
of $G_\Q$ which is $4$ (see \cite{Langlands}), we obtain that
\begin{equation}\label{eqdone}
\displaystyle{\lim_{X\to\infty}\displaystyle\frac{\displaystyle
    \sum_{\substack{E\in
        F\\H'(E)<X}}\#\{\sigma\in S_4(E):\sigma^2\neq 1\}}{\displaystyle\sum_{\substack{E\in
        F\\H'(E)<X}}1}}=4.
\end{equation}

Now, for any elliptic curve $E$ over $\Q$, the short exact sequence
$$
0\to E[2]\to E[4]\to E[2]\to 0
$$
yields the long exact sequence
$$
0\to E[2](\Q)\to E[4](\Q)\to E[2](\Q)\to H^1(\Q,E[2])\to H^1(\Q,E[4]).
$$
  Therefore, if $E$ has no nontrivial rational $2$-torsion points, then
the group $H^1(\Q,E[2])$ injects into $H^1(\Q,E[4])$. This implies
that $S_2(E)$ injects into $S_4(E)$, and thus 
$$\#S_4(E)=\#\{\sigma\in S_4(E):\sigma^2\neq 1\}+\#S_2(E).$$

The number of elliptic curves over $\Q$ having nontrivial rational
$2$-torsion and height less than $X$ is negligible, i.e., is
$o(X^{5/6})$. That the sum of the sizes of the $4$-Selmer groups of
such elliptic curves is negligible follows from
Proposition~\ref{propstrongnottot}.  Since we have shown in
\cite[Theorem~3.1]{BS} that the average size of the $2$-Selmer group of
elliptic curves in any large family $F$ is equal to $3$, we obtain from
(\ref{eqdone}) that
\begin{equation*}
  \displaystyle{\lim_{X\to\infty}\displaystyle\frac{\displaystyle
    \sum_{\substack{E\in
        F\\H'(E)<X}}\#S_4(E)}{\displaystyle\sum_{\substack{E\in
        F\\H'(E)<X}}1}}=4+3=7.
\end{equation*}
This concludes the proof of Theorem \ref{ellipall}
(and hence also of
%. 
Theorems \ref{mainellip} and \ref{ellipcong}).

\vspace{.1in} Finally, to obtain Theorem \ref{4lift}, we note that for
an elliptic curve $E$ over $\Q$ with no rational 2-torsion, if the
$4$-Selmer group $S_4(E)$ is isomorphic to
$(\Z/4\Z)^a\times(\Z/2\Z)^b$, then the $2$-Selmer group $S_2(E)$ is
isomorphic to $(\Z/2\Z)^{a+b}$; the number of 2-Selmer elements that
are not in the image of the $\times 2$ map from $S_4(E)$ to $S_2(E)$
is thus $2^{a+b}-2^a$ in this case.  To prove Theorem~\ref{4lift}, we
wish to determine a lower bound on the liminf of the average of
$2^{a+b}-2^a$ over all elliptic curves $E$ over~$\Q$ (having trivial
rational 2-torsion), when these elliptic curves are 
ordered by height.  Equivalently, we wish to
determine an upper bound on the limsup of the average size of $2^a$.

We have proven that the average number of order 4 elements in the
4-Selmer groups of these elliptic curves is 4, i.e., the average size
of $(4^a-2^a)2^b$ is 4.  It follows that the limsup of the average
size of $4^a-2^a$ is at most 4.  Since $5\cdot 2^a-8\leq 4^a-2^a$ for
all integers $a>0$, we conclude that the limsup of the average size of
$2^a$ is at most $12/5$.  Hence the liminf of the average
size of $2^{a+b}-2^a$ is at least $3-12/5=3/5$; this completes the
proof of Theorem~\ref{4lift}. (We note that the proof also naturally yields
a distribution of $2$- and $4$-Selmer groups---for which the 
average sizes of these groups are given by $3$
and $7$, respectively---that achieves the bound of 3/5; hence the bound 
of $3/5$ in Theorem~\ref{4lift} is in fact the best possible given these two
constraints.)  

As a consequence, we see that a proportion of at least $(3/5)/3=1/5$
of 2-Selmer elements of elliptic curves $E$ over $\Q$, when ordered
by height, do not lift to 4-Selmer elements; i.e., we have proven that
at least a fifth of
all 2-Selmer elements yield
nontrivial 2-torsion elements in the corresponding Tate--Shafarevich groups.

\subsection*{Acknowledgments}

We thank John Cremona, Christophe Delaunay, Tom Fisher, Wei Ho, Bjorn
Poonen, Michael Stoll, Jerry Wang, Kevin Wilson, and Melanie Wood for
helpful conversations. The first author was partially supported
by NSF Grant~DMS-1001828.

\end{document}